\documentclass[11pt, reqno]{amsart}
\usepackage{color}
\usepackage{amsmath}
\usepackage{amssymb}
\usepackage{amsfonts}
\usepackage{mathrsfs}
\usepackage{amsbsy}
\usepackage{relsize}
\usepackage[colorlinks, linkcolor=red, citecolor=blue, urlcolor=blue, pagebackref, hypertexnames=false]{hyperref}
\usepackage[hyperpageref]{backref}
\usepackage{geometry}
\newcommand{\R}{\mathbb R}

\newtheorem{theo}{Theorem}[section]
\newtheorem{prop}[theo]{Proposition}
\newtheorem{lem}[theo]{Lemma}
\newtheorem{rem}[theo]{Remark}

\newtheorem{cor}[theo]{Corollary}
\numberwithin{equation}{section}
\title[Relaxation to equilibrium in the thin film equation]{Relaxation to equilibrium in the  one-dimensional thin-film equation with partial wetting}

\author[ M. Majdoub, N. Masmoudi, and S. Tayachi]{Mohamed Majdoub, Nader Masmoudi, and Slim Tayachi}

\address[M. Majdoub]{Department of Mathematics, College of Science, Imam Abdulrahman Bin Faisal University, P. O. Box 1982, Dammam, Saudi Arabia}
\address{Basic and Applied Scientific Research Center, Imam Abdulrahman Bin Faisal University, P.O. Box 1982, 31441, Dammam, Saudi Arabia}
\email{\sl mmajdoub@iau.edu.sa}

\address[N. Masmoudi]{Courant Institute of Mathematical Sciences, New York, USA and Department of Mathematics, NYU Abu-Dhabi, UAE}
\email{\sl masmoudi@cims.nyu.edu}

\address[S. Tayachi]{Universit\'e de Tunis El Manar, Facult\'e des Sciences de Tunis, D\'epartement de math\'ematiques, Laboratoire \'equations aux d\'eriv\'ees partielles (LR03ES04), 2092 Tunis, Tunisie}
\email{\sl slim.tayachi@fst.rnu.tn}

\begin{document}
\begin{abstract}  We investigate the large time behavior of compactly supported solutions for a one-dimensional thin-film equation with linear mobility in the regime of partial wetting. We show  the stability  of  steady state solutions. The proof uses the Lagrangian coordinates. Our method is to establish and exploit differential relations between the energy and the dissipation as well as  some interpolation inequalities. Our result is different from earlier results because here we consider solutions with finite mass.
\end{abstract}

%@@@@@@@@@@@@@@@@@@@@@@@@@@@@@@@@@@@@%@@@@@@@@@@@@@@@@@@@@@@@@@@@@@@@@@@@@%@@@@@@@@@@@@@@

\subjclass[2010]{35B40, 35K25, 35K45, 35K65, 35A09, 35B35, 35B65, 35C06, 35R35, 76A20, 76D08, 76D27}

%@@@@@@@@@@@@@@@@@@@@@@@@@@@@@@@@@@@@%@@@@@@@@@@@@@@@@@@@@@@@@@@@@@@@@@@@@%@@@@@@@@

\keywords{Degenerate parabolic equations, fourth-order equations,
nonlinear parabolic equations, free boundary problems, stability, asymptotic behavior of solutions, classical solutions, thin fluid films, lubrication theory, Hele–Shaw flows, partial wetting, Lagrangian variables, decay estimate.}

%@@@@@@@@@@@@@@@@@@@@@@@@@@@@@@@@@@@@%@@@@@@@@@@@@@@@@@@@@@@@@@@@@@@@@@@@@%@@@@@@@@@@@@@@
\maketitle
\date{today}

%@@@@@@@@@@@@@@@@@@@@@@@@@@@@@@@@@@@@%@@@@@@@@@@@@@@@@@@@@@@@@@@@@@@@@@@@@%@@@@@@@@@@@@@@@

\section{Introduction}
\setcounter{equation}{0}

Consider the following one-dimensional fourth-order nonlinear degenerate parabolic equation
\begin{subequations}\label{tfe}
\begin{align}
u_t+\left(u u_{xxx}\right)_x&=0, \quad\quad\quad \mbox{for $t>0$ and $x\in \left(\lambda_-(t),\lambda_+(t)\right)$},\label{tfe1}\\
u (t,\lambda_{\pm}(t))& = 0, \quad\quad\quad \mbox{for $t > 0$},\label{tfe2}\\
|u_x(t,\lambda_{\pm}(t))|&=1,\quad\quad\quad \mbox{for $t > 0$},\label{tfe3}\\
\lim_{x\to \lambda_{\pm}(t)}\,u_{xxx}(t,x)&=V,\quad\quad\quad \mbox{for $t > 0$}, \label{tfe4}
\end{align}
\end{subequations}
where $\lambda_{\pm} : (0,\infty)\to(0,\infty)$ represent the support of $u$, $u>0$ on $(\lambda_-, \lambda_+)$ and $V$ is the velocity of the moving boundary. We supplement the problem \eqref{tfe} with the initial data $u(0,x)=u_0(x)$ supported in $(\lambda^0_-, \lambda^0_+)$ and satisfying \eqref{tfe3}.

Equation \eqref{tfe1} arises as the particular case of the thin-film equation in the  Hele-Shaw setting \cite{Bertozzi98, Myers98, Otto-CPDE}. It describes the pinching of thin necks in a Hele-Shaw cell. The function $u = u(t,x)\geq 0$ represents the height of a two-dimensional viscous thin film on a one-dimensional flat solid as a function of time $t > 0$ and the lateral variable $x$. One can rigorously derive equation~\eqref{tfe} from the Hele-Shaw
cell in the regime of thin films and where the dominating effects are surface tension and viscosity only \cite{go.2002,km.2012,km.2013}.

Equation \eqref{tfe} is a particular case of the thin-film equation
 \begin{equation}\label{tfen}
u_t + \big(u^n u_{xxx}\big)_x  = 0  \quad \mbox{for $t > 0$ and $x \in \R$},
\end{equation}
where $n>0$ is the mobility exponent. See \cite{ODB} for a physical explanation of the equation \eqref{tfen}. For $n\in (0,3)$ and in the complete wetting regime, the existence of self-similar source-type solution for \eqref{tfen} has been established  in \cite{BPW} by ODE arguments.  For $n=1$, the authors of \cite{BPW} prove the uniqueness in the class of self-similar solutions. See \cite[Lemma 6.3, p. 231]{BPW}. Recently, the uniqueness of source type solutions is proved in \cite{MMT2018} for a larger class. For $n=2$, well-posedness results are established in \cite{Knup2011} for \eqref{tfen} in the case of partial wetting regime.

The existence of weak solutions to \eqref{tfe} was investigated in \cite{Otto-CPDE}. More recently, the existence and uniqueness of classical solutions to \eqref{tfe} was shown in \cite[Theorem 4, p. 607]{km.2012}. The asymptotic behavior of solutions to thin film equations with prescribed contact angle has not been considered neither in \cite{Otto-CPDE} nor in \cite{km.2012}. However, the asymptotic behavior in the complete wetting regime was investigated in many papers. See, among many, \cite{CU2007, CU2014, CT2002,GKO, G,GIM} and references therein.

In \cite{Ess}, the author studies the following problem for the thin-film equation in the partial wetting regime on a half-axis, with a single contact point
\begin{equation}
\label{Ess}
h_t+\left(h h_{xxx}\right)_x=0,\quad\quad \mbox{in}\quad \left(\chi(t),\infty\right),
\end{equation}
$$
h (t,\chi(t)) = 0, \quad h_x(t,\chi(t))=1,\quad
h_{xxx}(t,\chi(t))=\dot\chi(t),\quad \mbox{for}\,\, t>0.
$$
Using a strategy inspired by \cite{OW2014}, the author of \cite{Ess} proves the stability of the steady state given by $h_0(x)=\max\{x,0\}$ for initial data close, in some sense, to $h_0$. See \cite[Theorem 1.2, p. 352]{Ess}. We would like to point out that only the case where the free boundary is given by a single contact point was considered in \cite{Ess}. As pointed out in \cite{Ess}, in contrast to \eqref{tfe}, solutions to \eqref{Ess} do not satisfy conservation of mass or even have finite mass.

In this paper we are interested in the case where the free boundary is given by two contact points at every time $t$. Solutions of \eqref{tfe} preserve mass and center of mass defined respectively by
\begin{equation}
\label{M}
\int_{\lambda_-(t)}^{\lambda_+(t)}\,u(t,x)\,dx=\int_{\lambda_-^0}^{\lambda_+^0}\,u(0,x)\,dx:=M\quad\mbox{for any}\quad t> 0;
\end{equation}
and
\begin{equation}
\label{CM}
\int_{\lambda_-(t)}^{\lambda_+(t)}\,xu(t,x)\,dx=\int_{\lambda_-^0}^{\lambda_+^0}\,xu(0,x)\,dx:=\mu\quad\mbox{for any}\quad t> 0.
\end{equation}
Note that if $u$ has a mass $M> 0$ and a center of mass $\mu\in\R$, then $\bar{u}(t,x)=u(t,x+\frac{\mu}{M})$ has the same mass but with center of mass equal to zero. Without loss of generality, we are going to suppose that $M> 0$ and $\mu\in\R$ are given.

Our main goal is the study of the stability of stationary solutions to \eqref{tfe}. The equation \eqref{tfe} possesses a family of stationary solutions given by
$$
u_{\alpha_-, \alpha_+}(x)=\frac{(x-\alpha_-)(\alpha_+-x)}{\alpha_+-\alpha_-},\quad \alpha_-<\alpha_+,\quad x\in(\alpha_-,\alpha_+).
$$
By choosing $M=2/3$ and $\mu=0$, we obtain $\alpha_-=-1$ and $\alpha_+=1$. So, we will investigate the stability of the stationary solution
$$
u^\infty(x)=\frac{1}{2}\,\left(1-x^2\right)_{+},
$$
where $a_+=\max\{a,0\}.$
To this end, let us define the mass Lagrangian variable as follows: for given $t\geq 0$,
$$
y\in (-1,1)\longmapsto x=Z(t,y)\in \left(\lambda_-(t), \lambda_+(t)\right),
$$
such that
\begin{equation}\label{MLV}
 \int_{\lambda_-(t)}^{Z(t,y)}\, u(t,x)\,dx=\frac{1}{2}\int_{-1}^y\,(1-x^2)\,dx,
\end{equation}
where $\lambda_{\pm}(0)=\lambda^0_{\pm}.$ The relation \eqref{MLV} together with \eqref{tfe} defines the function $Z$. One advantage of this transformation is that in the new coordinates the boundary is fixed to $y=\pm 1$. Moreover, the transformation \eqref{MLV} can be seen as a perturbation of the stationary solution $u^\infty$. In addition, the stationary solution in these coordinates is given by
$$
\label{Zinfty}
Z^\infty(t,y)=y,
$$
which is a linear function in $y$.

Before stating our result, we introduce the energy
\begin{equation}
\label{En}
E(t)=\frac{1}{2}\int_{-1}^1\,\Big|\partial_y\left(\frac{\tilde{u}(t,y)}{u^\infty(y)}\right)\Big|^2\,dy,
\end{equation}
where $\tilde{u}(t,y)=u(t, Z(t,y))$. We also define $Z_0(y):=Z(0,y)$ and $\tilde{u}_0(y)=u_0(Z_0(y)).$

Our main result can be stated as follows.
\begin{theo}
\label{MainThm}
Let $u$ be a global smooth solution to \eqref{tfe} with initial data $u_0$ satisfying $M(u_0)=2/3$ and $\mu(u_0)=0$. There exists a constant $c_0>0$ such that, if
\begin{equation}
\label{E0}
E(0)< c_0,
\end{equation}
then there exists $\gamma>0$ such that
\begin{equation}
\label{Et}
E(t)\leq E(0)\,{\rm e}^{-\gamma t},\quad t\geq 0.
\end{equation}
\end{theo}
\begin{rem}
{\rm Equation \eqref{tfe} is invariant under the scaling
$$u_\kappa(t,x)=\kappa^{-1}u(\kappa^3 t,\kappa x),\; \kappa>0.$$
The Lagrangian coordinate $Z_\kappa$ verifies
$$Z_\kappa(t,y)=\kappa^{-1}Z(\kappa^3 t,\kappa y),\; -1<\kappa y<1,$$ and
$$g_k(t,y)=g(\kappa^3 t,\kappa y),$$
where $g$ is given by \eqref{defg} below. This leads to
$$E_\kappa(0)=E_1(0)=E(0).$$ Then  the condition \eqref{E0} is invariant under the above scaling and hence it is relevant.}
\end{rem}
From the above theorem we easily derive the following.
\begin{cor}
\label{u0}
Suppose \eqref{E0} is fulfilled. Then it follows that
\begin{equation}
\label{u00}
\|\tilde{u}_0-u^\infty\|_{L^\infty(-1,1)}\lesssim c_0,
\end{equation}
and
\begin{equation}
\label{u000}
\|\tilde{u}(t)-u^\infty\|_{L^\infty(-1,1)}\lesssim {\rm e}^{-\frac{\gamma}{2} t},\quad t\geq 0.
\end{equation}
\end{cor}

An important consequence of the decay estimate \eqref{Et} is that it imply the desired convergence of $\lambda_{\pm}$ to $\pm 1$ as stated below.
\begin{cor}
\label{lamb}
Suppose \eqref{E0} is fulfilled. Then, for $t$ sufficiently large, we have
\begin{equation}
\label{cvlam}
|\lambda_+(t)-\lambda_-(t)-2|\lesssim {\rm e}^{- \frac{\gamma}{2} t},
\end{equation}
and
\begin{equation}
\label{lambda0}
|\lambda_+(t)+\lambda_-(t)|\lesssim {\rm e}^{- \frac{\gamma}{2} t},
\end{equation}
where $\gamma$ is as in \eqref{Et}. In particular, for $t$ large, we have
\begin{equation}
\label{lambdapm}
|\lambda_+(t)-1|+|\lambda_-(t)+1|\lesssim {\rm e}^{- \frac{\gamma}{2} t}.
\end{equation}
\end{cor}
We also obtain from Theorem \ref{MainThm} the convergence of the volumetric coordinates $Z(t,y)$ to $y$ with the same decay rate as in \eqref{Et}.
\begin{cor}
\label{Zt}
Suppose \eqref{E0} is fulfilled. Then, for $t$ sufficiently large, we have
\begin{equation}
\label{CVZ}
\|\partial_y(Z(t)-Z^\infty)\|_{L^\infty(-1,1)}\lesssim {\rm e}^{-\frac{\gamma}{2} t},
\end{equation}
where $\gamma$ is as in \eqref{Et}.
\end{cor}

The rest of this paper is organized as follows. In the next section, we reformulate our problem and derive a first energy estimate. In Section $3$, we recall and establish  some preliminaries and useful tools. Section $4$ is devoted to the energy estimate. In the last section we give the proof of the main results. We will write $A\lesssim B$ if there exists a constant $0<C<\infty$ such that $A\leq C B$, and $A\approx A_1+A_2+\cdots A_N$ if there exist constants $c_1,c_2,\cdots c_N\in\R$ such that $A=c_1 A_1+c_2 A_2+\cdots c_n A_N$. Finally, we denote the norm in Lebesgue space $L^p$ by $\|\cdot\|_p$, $1\leq p\leq \infty$.

\section{Reformulation of the Problem}
In this section we reformulate the problem \eqref{tfe}. From \eqref{MLV} we deduce the following.
\begin{prop}
\label{Z}
Let $Z$ be the mass Lagrangian variable defined by \eqref{MLV}. Then, we have
\begin{itemize}
\item[(i)] $Z(t,\pm 1)=\lambda_{\pm}(t).$
\item[(ii)] $Z_t(t,y)=u_{xxx}(t,Z(t,y)).$
\item[(iii)] $Z_y(t,y)=\frac{u^\infty(y)}{u(t,Z(t,y))}.$
\end{itemize}
\end{prop}
\begin{proof}
Part (i) follows immediately from the definition \eqref{MLV}. To prove (ii) we differentiate \eqref{MLV} with respect to variable $t$ to obtain
$$
0= \int_{\lambda_-(t)}^{Z(t,y)}\, u_t(t,x)\,dx+Z_t(t,y)u(t,Z)+\frac{d\lambda_-(t)}{dt}\,u(t,\lambda_-(t)).
$$
Using boundary condition \eqref{tfe2} and equation \eqref{tfe1}, we get
$$
0=u(t,Z)\Big(Z_t(t,y)-u_{xxx}(t,Z)\Big),
$$
{since $u\,u_{xxx}=0$ at $x=\lambda_{\pm}(t)$.} This obviously leads to (ii).

Now we differentiate \eqref{MLV} with respect to variable $y$ to obtain
$$
Z_y(t,y)\,u(t, Z(t,y))=u^\infty(y),
$$
which is exactly  (iii).
\end{proof}
\begin{rem}$\;${\rm
\begin{itemize}
\item[1)] By choosing $y=1$ in \eqref{MLV} and using the fact that $Z(t,1)=\lambda_+(t)$ we get
$$M=\int_{\lambda_-(t)}^{\lambda_+(t)}\,u(t,x)\,dx=\int_{-1}^{1}u^\infty(x)\,dx=\frac{2}{3}.$$
\item[2)] Making the change of variable $x=Z(t,y)$ and using (iii) of the previous proposition, we see that the center of mass given by \eqref{CM} reads
\begin{equation}
\label{CMZ}
\mu=\int_{\lambda_-(t)}^{\lambda_+(t)}\,x\,u(t,x)\,dx=\frac{1}{2}\int_{-1}^1\, \left(1-y^2\right)\,Z(t,y)\,dy.
\end{equation}
\end{itemize}}
\end{rem}

In order to derive an evolution equation in the Lagrangian coordinates, observe that by the previous proposition we have
\begin{equation}
\label{uZinfty}
Z_y(t,y)\,u(t, Z(t,y))=u^\infty(y).
\end{equation}
The relation \eqref{uZinfty} suggests to define
\begin{equation}
\label{G}
G=\frac{1}{Z_y}.
\end{equation}
The evolution equation satisfied by $G$ is given in the following proposition.
\begin{prop}
\label{GG}
The function $G$ defined by \eqref{G} solves
\begin{subequations}\label{tfG}
\begin{align}
G_t+G^2\partial_y\Big(\left(G\partial_y\right)^3\left(u^\infty\,G\right)\Big)&=0, \quad\quad\quad \mbox{for $t>0$ and $y\in \left(-1,1\right)$},\label{tfG1}\\
G(t,\pm1)& =1, \quad\quad\quad \mbox{for $t > 0$}.\label{tfG2}
\end{align}
\end{subequations}
\end{prop}
\begin{proof}
Differentiate \eqref{G} with respect to $t$ yields
\begin{equation}
\label{Gt}
G_t=-\frac{Z_{ty}}{Z_y^2}=-G^2\,Z_{ty}.
\end{equation}
Since, by Proposition \ref{Z} (iii), $u^\infty(y)\,G(t,y)=u(t, Z(t,y))$, it follows that $\partial_y(u^\infty\,G)=Z_y\,u_x=\frac{1}{G}\,u_x$. Therefore
$$
G\,\partial_y\left(u^\infty\,G\right)=u_x(t,Z).
$$
Similarly we obtain that
$$
\left(G\,\partial_y\right)^2\,\left(u^\infty\,G\right)=u_{xx}(t,Z)\quad\mbox{and}\quad \left(G\,\partial_y\right)^3\,\left(u^\infty\,G\right)=u_{xxx}(t,Z).
$$
Using Part (ii) in Proposition \ref{Z} and \eqref{Gt}, we infer
\begin{eqnarray*}
G_t&=&-G^2\,Z_{ty}\\
&=&-G^2\,\partial_y\,Z_t\\
&=&-G^2\,\partial_y\,\left(u_{xxx}(t,Z)\right)\\
&=&-G^2\,\partial_y\,\Big(\left(G\,\partial_y\right)^3\,\left(u^\infty\,G\right)\Big)
\end{eqnarray*}
This leads to \eqref{tfG1}. To prove \eqref{tfG2} we use the L'H\^opital rule to deduce that
$$
Z_y(t,1)=\frac{(u^\infty)'(1)}{Z_y(t,1)\,u_x(t,\lambda_+(t))}.
$$
Hence $(Z_y(t,1))^2=1$. By \eqref{uZinfty}, $Z_y\geq 0$. Then, using \eqref{G} we get \eqref{tfG2}. Similarly for $y=-1.$
\end{proof}
Since we want to study perturbations of $G^\infty=\frac{1}{Z^\infty_y}=1$, we set
\begin{equation}
\label{defg}
G=1+g.
\end{equation}
Refereeing to \eqref{tfG}, we obtain that
\begin{subequations}\label{tfgg}
\begin{align}
g_t+{\mathcal L}\,g&={\mathcal N}(g), \quad\quad\quad \mbox{for $t>0$ and $y\in \left(-1,1\right)$},\label{tfgg1}\\
g(t,\pm1)& =0, \quad\quad\quad \mbox{for $t > 0$},\label{tfgg2}
\end{align}
\end{subequations}
where
\begin{equation}
\label{L}
{\mathcal L}\,g=-10\partial_y^2\,g-5y\,\partial_y^3\,g+\frac{1-y^2}{2}\,\partial_y^4\,g,
\end{equation}
and
\begin{eqnarray}
\nonumber
{\mathcal N}(g)&=&-25(1+g)^4\,\left(\partial_y\,g\right)^2-10\left[(1+g)^5-1\right]\,\partial_y^2\,g\\
&-&\nonumber15y(1+g)^3\left(\partial_y\,g\right)^3-30y(1+g)^4\,\partial_yg\,\partial_y^2\,g\\
&-&\label{NN}5y\left[(1+g)^5-1\right]\partial_y^3\,g+\frac{1-y^2}{2}(1+g)^2\,(\partial_y\,g)^4\\
&+&\nonumber\frac{11}{2} (1-y^2)(1+g)^3 (\partial_y\,g)^2 \partial_y^2\,g\\
&+&\nonumber2(1-y^2)(1+g)^4\,(\partial_y^2\,g)^2+\frac{7}{2}(1-y^2)(1+g)^4\,\partial_y\,g\,\partial_y^3\,g\\
&+&\nonumber\frac{1-y^2}{2}\left[(1+g)^5-1\right]\partial_y^4\,g.
\end{eqnarray}

Using \eqref{uZinfty}, \eqref{G} and \eqref{defg}, we may rewrite the energy given by \eqref{En} as
\begin{equation}
\label{E}
E(t)=\frac{1}{2}\|\partial_y\,g(t)\|_{2}^2.
\end{equation}
We also introduce the dissipation
\begin{equation}
\label{D}
D(t)=8\|\partial_y^2\,g(t)\|_{2}^2+\frac{1}{2}\|\sqrt{1-y^2}\,\partial_y^3\,g(t)\|_{2}^2,
\end{equation}
as well as the following quantities
\begin{eqnarray*}
{\mathbf I}_1&=& 25\int_{-1}^1\,(1+g)^4\, (\partial_y\,g)^2\,\partial_y^2\,g\,dy,\;\;\; {\mathbf I}_2= 10\int_{-1}^1\,\left[(1+g)^5-1\right]\,\left(\partial_y^2\,g\right)^2\,dy,\\
{\mathbf I}_3&=&15\int_{-1}^1\,y(1+g)^3\,\left(\partial_y\,g\right)^3\,\partial_y^2\,g\,dy,\;\;\;
{\mathbf I}_4= 30\int_{-1}^1\,y (1+g)^4\,\partial_y\,g\left(\partial_y^2\,g\right)^2\,dy,\\
{\mathbf I}_5&=& 5\int_{-1}^1\,y\left[(1+g)^5-1\right]\,\partial_y^3\,g\,\partial_y^2\,g\,dy,\;\;\;{\mathbf I}_6=-\frac{1}{2}\int_{-1}^1\,(1-y^2) (1+g)^2\,\left(\partial_y\,g\right)^4\,\partial_y^2\,g\,dy,\\
{\mathbf I}_7&=& -\frac{11}{2}\int_{-1}^1\,(1-y^2)(1+g)^3\,\left(\partial_y\,g\right)^2\,\left(\partial_y^2\,g\right)^2\,dy,\;\;\; {\mathbf I}_8=-2\int_{-1}^1\,(1-y^2)\,(1+g)^4\,\left(\partial_y^2\,g\right)^3\,dy,\\
{\mathbf I}_9&=& -\frac{7}{2}\int_{-1}^1\,(1-y^2)\,(1+g)^4\, \partial_y\,g\,\partial_y^3\,g\,\partial_y^2\,g\,dy,\;\;\;
{\mathbf I}_{10}=-\frac{1}{2}\int_{-1}^1\,(1-y^2)\left[(1+g)^5-1\right]\,\partial_y^4\,g\,\partial_y^2\,g\,dy.
\end{eqnarray*}
Our aim now is to obtain a first energy estimate.
We have obtained the following.
\begin{prop}
\label{EE1}
Le $E$ and $D$ be given by \eqref{E} and \eqref{D} respectively. Then
\begin{equation}
\label{Energy}
\frac{dE}{dt}+D\leq  \sum_{k=1}^{10}\,{\mathbf I}_k.
\end{equation}
\end{prop}
\begin{proof}
Multiplying \eqref{tfgg1} by $-\partial_y^2\,g$ and integrating in $y\in(-1,1)$ we get, using \eqref{tfgg2}
\begin{equation*}
\frac{dE}{dt}+D=-2\left[\left(g_{yy}(1)\right)^2+\left(g_{yy}(-1)\right)^2\right]+ \sum_{k=1}^{10}\,{\mathbf I}_k.
\end{equation*}
In fact, since $g(t,\pm 1)=0$ for all $t>0$, we have
\begin{eqnarray*}
-\int_{-1}^1g_tg_{yy}dy&=&-[g_tg_y]_{-1}^1+\int_{-1}^1(\partial_tg_y)g_ydy\\&=&\frac{1}{2}\partial_t\left(\int_{-1}^1g_y^2dy\right):=\frac{dE}{dt}.
\end{eqnarray*}
Also,
\begin{eqnarray*}-\int_{-1}^1\frac{1-y^2}{2}\partial_y^4 g\partial_y^2gdy&=&-\left[\frac{1-y^2}{2}\partial_y^3g\partial_y^2g\right]_{-1}^1+\int_{-1}^1\frac{1-y^2}{2}\left(\partial_y^3g\right)^2dy-\int_{-1}^1y\partial_y^2g\partial_y^3gdy\\ &=&\int_{-1}^1\frac{1-y^2}{2}\left(\partial_y^3g\right)^2dy-\int_{-1}^1y\partial_y^2g\partial_y^3gdy
\end{eqnarray*}
Hence
\begin{eqnarray*}5\int_{-1}^1y\partial_y^3g \partial_y^2gdy-\int_{-1}^1\frac{1-y^2}{2}\partial_y^4 g\partial_y^2 gdy&=&\int_{-1}^1{\frac{1-y^2}{2}}\left(\partial_y^3 g\right)^2dy+4\int_{-1}^1y\partial_y^2 g\partial_y^3 gdy\\&=&\int_{-1}^1
{\frac{1-y^2}{2}}\left(\partial_y^3 g\right)^2dy+4\int_{-1}^1y \frac{1}{2}\partial_y\left((\partial_y^2 g)^2\right)dy\\&=&\int_{-1}^1{\frac{1-y^2}{2}}\left(\partial_y^3 g\right)^2dy+2[y\left(\partial_y^2g\right)^2]_{-1}^1-2\int_{-1}^1\left(\partial_y^2 g\right)^2dy.
\end{eqnarray*}
Then
\begin{eqnarray*}&&10\int_{-1}^1\left(\partial_y^2g\right)^2dy+5\int_{-1}^1y\partial_y^3g \partial_y^2gdy-\int_{-1}^1\frac{1-y^2}{2}\partial_y^4 g\partial_y^2gdy=\\&&\int_{-1}^1\frac{1-y^2}{2}\left(\partial_y^3g\right)^2dy+2[y\left(\partial_y^2g\right)^2]_{-1}^1+8\int_{-1}^1\left(\partial_y^2g\right)^2dy=\\&&2\left[y\left(\partial_y^2g\right)^2\right]_{-1}^1+D(t),
\end{eqnarray*}
and we get \eqref{Energy}.

\end{proof}
\section{Useful tools}
In this section, we recall some known and useful tools. Then, we use them to obtain crucial estimates needed in our proofs.
\begin{prop}[{\tt Poincar\'e's inequality}]
\label{PI}
Suppose $I=(a,b)\subset\R$ is a bounded interval and $1\leq\,p,\,q\leq\infty$. Then, for every $u\in\,W^{1,p}_0(I)$, we have
\begin{equation}
\label{PoIn}
\|u\|_{q}\leq\,|I|^{\frac{1}{q}-\frac{1}{p}+1}\,
\|u'\|_{p}.
\end{equation}
\end{prop}
\begin{proof}
For $x\in I$ we have $u(x)=\int_a^x\,u'(t)\,dt$. Hence, for all $x\in I$
\begin{equation}
\label{PC1}
|u(x)|\leq \int_{I}\,|u'(t)|\,dt\leq |I|^{1-\frac{1}{p}}\,\|u'\|_p.
\end{equation}
This proves \eqref{PoIn} for $q=\infty$. Suppose now $q<\infty$. Then, by \eqref{PC1}, we obtain that
$$
\int_{I}\,|u(x)|^q\,dx\leq |I|^{1+q-\frac{q}{p}}\,\|u'\|_p^q.
$$
This leads to \eqref{PoIn} as desired.
\end{proof}
\begin{rem}
{\rm For $u\in W^{1,p}(I)$ we introduce
$$
\mathcal{Z}(u):=\Big\{\, x\in \bar{I};\;\;\; u(x)=0\,\Big\}.
$$
This definition makes sense since $W^{1,p}(I)\hookrightarrow C(\bar{I})$. We also define
$$
\tilde{W}^{1,p}(I):=\Big\{\, u\in W^{1,p}(I);\;\;\; \mathcal{Z}(u)\neq \emptyset\,\Big\}.
$$
Clearly $W^{1,p}_0(I)\subset \tilde{W}^{1,p}(I).$}
\end{rem}
A more general statement of Poincar\'e's inequality can be stated as follows.
\begin{prop}
\label{PIG}
Suppose $I=(a,b)\subset\R$ is a bounded interval and $1\leq\,p,\,q\leq\infty$. Then, for every $u\in\,\tilde{W}^{1,p}(I)$, we have
\begin{equation}
\label{PoInG}
\|u\|_{q}\leq\,|I|^{\frac{1}{q}-\frac{1}{p}+1}\,
\|u'\|_{p}.
\end{equation}
\end{prop}
From Proposition \ref{PIG} we deduce the following Poincar\'e-Wirtinger inequality.
\begin{prop}[{\tt Poincar\'e-Wirtinger's inequality}]
\label{PW}
Let $1\leq p, q\leq\infty$ and $I=(a,b)\subset\R$ a bounded interval. Then, for any $u\in W^{1,p}(I)$, we have
\begin{equation}
\label{PoWi}
\|u-\overline{u}\|_{q}\leq\,|I|^{\frac{1}{q}-\frac{1}{p}+1}\,\|u'\|_{p},
\end{equation}
where
$$
\overline{u}=\frac{1}{|I|}\,\int_{I}\,u(y)\,dy.
$$
\end{prop}
\begin{proof}
Let $u\in W^{1,p}(I)$ and define $v=u-\overline{u}$. Since $\int_{I}\,v=0$ and $v\in C(\bar{I})$, then $v\in \tilde{W}^{1,p}(I)$. Applying \eqref{PoInG} with $v$ we conclude the proof since $v'=u'$.
\end{proof}
Note that, if $u\in W^{2,p}(I)\cap W^{1,p}_0(I)$, then $\overline{u'}=0$. Hence, by applying \eqref{PoWi} with $u'$ instead of $u$, we obtain the following useful inequality.
\begin{cor}
Let $1\leq p, q\leq\infty$ and $I=(a,b)\subset\R$ a bounded interval. Then, for any $u\in W^{2,p}(I)\cap W^{1,p}_0(I)$, we have
\begin{equation}
\label{pw1}
\|u'\|_{q}\leq\,|I|^{\frac{1}{q}-\frac{1}{p}+1}\,\|u''\|_{p}.
\end{equation}
\end{cor}

We also recall the following Gagliardo-Nirenberg interpolation inequalities useful for our purpose. We refer to \cite{Brezis, Nir1, Nir2} for more general statements.
\begin{prop}
\label{GNI-B}
Suppose $I=(a,b)\subset\R$ is a bounded interval, $1\leq\,q<\infty$ and $1\leq\,r\leq\,\infty$. Then, there exists a constant $C=C(q,r)>0$ such that for every $u\in W^{1,r}(I)$, we have
\begin{equation}
\label{St1}
\|u\|_\infty\leq\, C\left(1+\frac{1}{|I|}\right)\|u\|_{W^{1,r}}^{\delta}\,\|u\|_q^{1-\delta},
\end{equation}
where $0<\,\delta\,\leq\, 1$ is defined by
\begin{equation}
\label{delta}
\delta\,\left(\frac{1}{q}+1-\frac{1}{r}\right)=\frac{1}{q}.
\end{equation}
\end{prop}
\begin{proof}
We start with the case $r>1$. We first suppose $u(a)=0$. Then
$$
|u(x)|^{\alpha-1}u(x)=\int_a^x\,G'(u(\tau))u'(\tau)\,d\tau,
$$
where $G(\tau)=|\tau|^{\alpha-1}\tau$ and $\alpha=\frac{1}{\delta}\in (1,\infty)$. It follows by H\"older's inequality that
\begin{eqnarray*}
|u(x)|^{\alpha}&\leq&\alpha\,\int_{I}\,|u(\tau)|^{\alpha-1}\,|u'(\tau)|\,d\tau\\
&\leq&\alpha\,\|u'\|_r\,\|u\|_{r'(\alpha-1)}^{\alpha-1}\\
&\leq&\alpha\,\|u'\|_r\,\|u\|_{q}^{\alpha-1},
\end{eqnarray*}
where we have used the fact that $r'(\alpha-1)=q$. Therefore
\begin{equation}
\label{case1}
\|u\|_\infty\leq\,\alpha^{\frac{1}{\alpha}}\,\|u'\|_r^{\frac{1}{\alpha}}\,
\|u\|_{q}^{1-\frac{1}{\alpha}}\leq\,\delta^{-\delta}\,\|u'\|_r^{\delta}\,\|u\|_{q}^{1-\delta}.
\end{equation}
We now turn to the case when $u(a)\neq 0$. Let $\eta\in C^1([a,b])$ defined by
\begin{eqnarray*}
 \eta(s)&=&\; \left\{
\begin{array}{cll}\frac{4}{|I|}(s-a)-\frac{4}{|I|^2}(s-a)^2 \quad&\mbox{if}&\quad
a\leq s\leq \frac{a+b}{2},\\\\ 1 \quad
&\mbox{if}&\quad \frac{a+b}{2}\leq s\leq b.
\end{array}
\right.
\end{eqnarray*}
Clearly $0\leq\eta\leq 1$ and $0\leq\eta'\leq \frac{4}{|I|}$. Applying \eqref{case1} respectively to $v(x):=\eta(x)u(x)$ and $w(x)=\eta(x)u(a+b-x)$, we get
\begin{equation}
\label{Case11}
|u(x)|\leq C\left(1+\frac{1}{|I|}\right)\|u\|_{W^{1,r}}^{\delta}\,\|u\|_{q}^{1-\delta},\;\;\;\forall\;\;x\in\left[\frac{a+b}{2}, b\right],
\end{equation}
and
\begin{equation}
\label{Case12}
|u(y)|\leq C\left(1+\frac{1}{|I|}\right)\|u\|_{W^{1,r}}^{\delta}\,\|u\|_{q}^{1-\delta},\;\;\;\forall\;\;y\in\left[a, \frac{a+b}{2}\right].
\end{equation}
Combining \eqref{Case11} and \eqref{Case12} we obtain the desired inequality \eqref{St1} when $r>1$. This finishes the proof since the case $r=1$ is trivial.
\end{proof}
\begin{rem}
{\rm A careful inspection of the proof shows that the constant $C$ appearing in \eqref{St1} can be taken as
$$
C(q,r)=C_0\,\left(1+q-\frac{q}{r}\right)^{\frac{1}{1+q-\frac{q}{r}}},
$$
where $C_0>0$ is an absolute constant.}
\end{rem}
\begin{prop}
\label{GN1D2}
Suppose $I=(a,b)\subset\R$ is a bounded interval, $1\leq\,q<\infty$, $q\leq\,p\,\leq \infty$ and $1\leq\,r\leq\,\infty$. Then, there exists a constant $C=C(|I|, p,q,r)>0$ such that for every $u\in\,W^{2,r}(I)\cap W^{1,r}_0(I)$, we have
\begin{equation}
\label{GNI6}
\|u'\|_{p}\leq\,C\|u''\|_{r}^{\theta}\,\|u'\|_{q}^{1-\theta},
\end{equation}
where $0\leq\,\theta\leq\,1$ is defined by
\begin{equation}
\label{theta}
\theta\left(\frac{1}{q}+1-\frac{1}{r}\right)=\frac{1}{q}-\frac{1}{p}.
\end{equation}
\end{prop}
\begin{proof}
The case $p=q$ is trivial. We will focus only on the case $q<p\leq\infty$. The proof will be divided into three steps.\\
\noindent{\bf Step 1.} We claim that
\begin{equation}
\label{St11}
\|v\|_p\leq\,C\|v\|_{W^{1,r}}^{\theta}\,\|v\|_q^{1-\theta},\;\;\forall\;\; v\in W^{1,r}(I),
\end{equation}
for some constant $C=C(|I|, p,q,r)>0$, where $\theta$ is given as in \eqref{theta}. Note that for $p=\infty$ we have $\theta=\delta$ and \eqref{St11} reduces to \eqref{St1}. To see \eqref{St11} for $p<\infty$ we write using \eqref{St1}
\begin{eqnarray*}
\|v\|_p^p&=&\int_{I}\,|v(x)|^{p-q}\,|v(x)|^q\,dx\\
&\leq&\|v\|_{\infty}^{p-q}\,\|v\|_q^q\\
&\leq&\left(C\left(1+\frac{1}{|I|}\right)\|v\|_{W^{1,r}}^{\delta}\,\|v\|_q^{1-\delta}\right)^{p-q}\,\|v\|_q^q.
\end{eqnarray*}
Therefore
$$
\|v\|_p\leq C^{1-\frac{q}{p}}\left(1+\frac{1}{|I|}\right)^{1-\frac{q}{p}}\,
\|v\|_{W^{1,r}}^{\delta(1-\frac{q}{p})}\,\|v\|_q^{(1-\delta)(1-\frac{q}{p})+\frac{q}{p}}.
$$
This leads to \eqref{St11} thanks to $\theta=\delta(1-\frac{q}{p})$ and $1-\theta=(1-\delta)(1-\frac{q}{p})+\frac{q}{p}$.\\

\noindent{\bf Step 2.} We claim that
\begin{equation}
\label{St111}
\|v\|_p\leq\,C\|v'\|_{r}^{\theta}\,\|v\|_q^{1-\theta},\;\;\forall\;\; v\in W^{1,r}(I)\;\;\;\mbox{with}\;\;\;\int_{I}\,v(x)\,dx=0,
\end{equation}
for some constant $C=C(|I|, p,q,r)>0$, where $\theta$ is given as in \eqref{theta}.\\
Since $\overline{v}=0$, we obtain by using \eqref{St11} and \eqref{PoWi} that
\begin{eqnarray*}
\|v\|_p=\|v-\overline{v}\|_p&\leq&C \|v-\overline{v}\|_{W^{1,r}}^{\theta}\,\|v-\overline{v}\|_q^{1-\theta}\\
&\leq& C\|v'\|_{r}^{\theta}\,\|v\|_q^{1-\theta}.
\end{eqnarray*}
\noindent{\bf Step 3.} Now we are ready to conclude the proof. Let $u\in\,W^{2,r}(I)\cap W^{1,r}_0(I)$. Then clearly $u'\in W^{1,r}(I)$ and $\displaystyle\int_{I}\,u'=0$. Applying \eqref{St111} to $v:=u'$ we get \eqref{GNI6}.
\end{proof}
\begin{rem}
{\rm Choose $q=r=2$ in Proposition \ref{GN1D2}, we see that $\theta=\frac{1}{2}-\frac{1}{p}$. The inequality \eqref{GNI6} takes the following form
\begin{equation}
\label{GNI22}
\|u'\|_p\lesssim \|u''\|_2^{\frac{1}{2}-\frac{1}{p}}\,\|u'\|_2^{\frac{1}{2}+\frac{1}{p}},
\end{equation}
for all $2\leq p\leq \infty$ and $u\in H^2(I)\cap H^1_0(I)$.}
\end{rem}
From the above inequalities we deduce the following estimates.
\begin{lem}
\label{ED}
Let $g=g(y)\in H^2(-1,1)\cap\, H^1_0(-1,1)$. Then, we have
\begin{itemize}
\item[(i)]\begin{equation}
\label{gE}
\|g\|_{\infty}\lesssim\, E^{1/2}.
\end{equation}
\item[(ii)] \begin{equation}
\label{ED1}
E\lesssim\,D.
\end{equation}
\item[(iii)] For every $2\leq\,p\leq \infty$,
\begin{equation}
\label{partialggg}
\|\partial_y\,g\|_p\lesssim\,E^{\frac{1}{4}+\frac{1}{2p}}\,D^{\frac{1}{4}-\frac{1}{2p}}.
\end{equation}
\item[(iv)]
\begin{equation}
\label{partialgsecond}
\|\sqrt{1-y^2}\partial_y^2\,g\|_2\lesssim\,E^{1/4}\,D^{1/4}.
\end{equation}

\item[(v)] {For every $2\leq\,p\,< \infty$,
\begin{equation}
\label{N}
\|\partial_y^2g\|_p\lesssim D^{1/2}.
\end{equation}
}
\item[(vi)] {For every $\epsilon>0$,
\begin{equation}
\label{eps}
\|(1-y^2)^\epsilon\,\partial_y^2\,g\|_\infty\lesssim\,D^{1/2}.
\end{equation}
}
\end{itemize}
\end{lem}
\begin{proof}\quad\\
\vspace{-0.5cm}
\begin{itemize}
\item[(i)] Applying \eqref{PoIn} with $q=\infty$ and $p=2$, we obtain \eqref{gE}.
\item[(ii)] The inequality \eqref{ED1} follows from \eqref{pw1} with $p=q=2$.
\item[(iii)] The inequality \eqref{partialggg} follows from \eqref{GNI6}.
\item[(iv)] We write, using integration by parts,
\begin{eqnarray*}\int_{-1}^1(1-y^2)\left(\partial^2_y g\right)^2 dy&=&\int_{-1}^1(1-y^2)\partial^2_y g \partial^2_y g dy\\&=&-\int_{-1}^1(1-y^2)\partial^3_y g \partial_y g dy+\int_{-1}^12y\partial^2_y g \partial_y g dy.\end{eqnarray*}
Then \eqref{partialgsecond} follows by the Cauchy-Schwarz inequality.
\item[(v)] To prove\eqref{N} we write,
\begin{eqnarray*}\partial_y^2g(y) & = &\partial_y^2g(0)+\int_0^y\partial_y^3g(z)dz\\ & = &\partial_y^2g(0)+\int_0^y \sqrt{1-z^2}\partial_y^3g(z)\frac{1}{\sqrt{1-z^2}}dz.\end{eqnarray*}
Since { $\sqrt{1-y^2}\geq \frac{\sqrt{3}}{2}$ for $y\in (-{1/2}, {1/2})$}, we obtain
\begin{eqnarray*}
|\partial_y^2g(0)|&\lesssim &\|\partial_y^2g\|_{L^\infty(-1/2,-1/2)}\\ &\lesssim & \|\partial_y^2g\|_{H^1(-1/2,-1/2)} \\ &\lesssim & \|\sqrt{1-y^2}\partial_y^3g\|_2+\|\partial_y^2g\|_2\\ &\lesssim& D^{1/2}.
\end{eqnarray*}
By Cauchy Schwarz's inequality, we get
\begin{eqnarray}
\nonumber
|\partial_y^2g(y)|&\lesssim & |\partial_y^2g(0)|+\left(\left|\int_0^y \left(\sqrt{1-z^2}\partial_z^3g(z)\right)^2dz\right|\right)^{1/2}\left(\left|\int_0^y\frac{1}{1-z^2}dz\right|
\right)^{1/2} \\
\label{epsi}
&\lesssim &  D^{1/2}+D^{1/2}\left(\frac{1}{2}\left|\log\left(\frac{1+y}{1-y}\right)\right|\right)^{1/2}
\end{eqnarray}
{ The desired inequality \eqref{N} follows thanks to the fact that $\left(\left|\log\left(\frac{1+y}{1-y}\right)\right|\right)^{1/2}\in L^p(-1,1)$ for any $2\leq p<\infty.$}
\item[(vi)] {Inequality \eqref{eps} can be deduced easily from \eqref{epsi} making use of $x^\alpha\,\log(x)\in L^\infty(0,a)$ for any positive constants $\alpha$ and $a$.}
\end{itemize}
\end{proof}

\begin{lem}
\label{Infi}
Let $m\geq 1$ be an integer and $g=g(y)\in H^2(-1,1)\cap\, H^1_0(-1,1)$. Then we have
\begin{equation}
\label{SUP}
\Big\|\frac{g^m}{1-y^2}\Big\|_{\infty}\lesssim\,E^{{m/2}-{1/4}}\,D^{{1/4}}.
\end{equation}
\end{lem}
\begin{proof}
Since $H^2 \hookrightarrow C^1$ and $g(\pm 1)=0$, we get by the mean value theorem that
\begin{eqnarray*}
\Big\|\frac{g^m}{1-y^2}\Big\|_{\infty}&\leq& \Big\|\frac{g^m}{1-y^2}\Big\|_{L^\infty(-1,0)}+\Big\|\frac{g^m}{1-y^2}\Big\|_{L^\infty(0,1)}\\
&\leq& 2m\|g\|_{\infty}^{m-1}\,\|\partial_y\,g\|_{\infty}.
\end{eqnarray*}
Therefore we obtain thanks to \eqref{gE} and \eqref{partialggg} (with $p=\infty$) that
\begin{eqnarray*}
\Big\|\frac{g^m}{1-y^2}\Big\|_{\infty}&\lesssim&E^{\frac{m-1}{2}}\,E^{{1/ 4}}\,D^{{1/ 4}}\\
&\lesssim&E^{\frac{2m-1}{4}}\,D^{{1/4}}.
\end{eqnarray*}
This finishes the proof.
\end{proof}

To obtain the desired decay estimate, we will use the following.
\begin{lem}
\label{Decay}
Let $E, D : [0,\infty)\to[0,\infty)$ be absolutely continuous functions such that
\begin{equation}
\label{Ineq1}
\frac{dE}{dt}+D\leq a\left(E^{\alpha}+E^{\beta}\,D\right),
\end{equation}
and
\begin{equation}
\label{Ineq2}
b E\leq  D,
\end{equation}
where $a, b>0$, $\alpha>1$ and $\beta> 0$.
Then there exists $\varepsilon>0$ and $\nu>0$ such that
\begin{equation}
\label{conc}
E(0)<\varepsilon\,\Longrightarrow\, E(t)\leq\,E(0)\,{\rm e}^{-\nu\,t},\; t\geq 0.
\end{equation}
\end{lem}

\begin{proof}
 Choose $\varepsilon>0$ small enough such that $1-a\varepsilon^\beta>0$ and
 \begin{equation}
 \label{nu}
    \nu:=b(1-a\varepsilon^\beta)-a\varepsilon^{\alpha-1}>0.
 \end{equation} Suppose that $E(0)<\varepsilon$ and define
    \begin{equation}
    \label{T}
    T^*=\sup\Big\{\, s\geq 0;\;\;\; E(t)\leq\varepsilon\quad\mbox{for all}\quad t\in [0,s]\,\Big\}.
    \end{equation}
By continuity of $E$ we deduce that $T^*>0$. For $0\leq t<T^*$, we have
\begin{eqnarray*}
\frac{dE}{dt}+(1-a\varepsilon^\beta)D&\leq& a E^\alpha,\\
&\leq& a \varepsilon^{\alpha-1}E.
\end{eqnarray*}
Using \eqref{Ineq2} we deduce that
$$
\frac{dE}{dt}+\nu\,E\leq 0,
$$
where $\nu$ is given by \eqref{nu}.
This finally leads to
$$
E(t)\leq E(0) {\rm e}^{-\nu\,t},\quad\mbox{for all}\quad 0\leq t<T^*.
$$
In particular $T^*=\infty$  and, for all $t\in [0,\infty)$, we have
$$
E(t)\leq\,E(0)\,{\rm e}^{-\nu\,t}.
$$
\end{proof}
From the above Lemma we deduce the following.
\begin{cor}
\label{DecayC}
Let $E, D : [0,\infty)\to[0,\infty)$ be absolutely continuous functions such that
\begin{equation}
\label{Ineq1C}
\frac{dE}{dt}+D\leq a\left(E^{\alpha_1}+E^{\alpha_2}+\left(E^{\beta_1}+E^{\beta_2}\right)\,D\right),
\end{equation}
and
\begin{equation}
\label{Ineq2C}
b E\leq  D,
\end{equation}
where $a, b>0$, $\alpha_1, \alpha_2>1$ and $\beta_1, \beta_2> 0$.
Then, for $c_0>0$ satisfying $1-a(c_0^{\beta_1}+c_0^{\beta_2})>0$ and
\begin{equation}
\label{gam}
\gamma:=b\left(1-a(c_0^{\beta_1}+c_0^{\beta_2})\right)-a(c_0^{\alpha_1-1}+c_0^{\alpha_2-1})>0,
\end{equation} we have
\begin{equation}
\label{concc}
E(0)<c_0\,\Longrightarrow\, E(t)\leq\,E(0)\,{\rm e}^{-\gamma\,t},\; t\geq 0.
\end{equation}
\end{cor}

\section{The energy estimate}

In this section we will see how to estimate all terms ${\mathbf I}_k$ to get the following energy inequality leading to the desired decay rate.
\begin{prop}
\label{EnEs}
We have
\begin{equation}
\label{Ener}
\frac{dE}{dt}+D\lesssim E^2+E^{10}+\left(E^{1/3}+E^{5/2}\right)\,D.
\end{equation}
\end{prop}

\begin{proof}
We have thanks to \eqref{Energy} that
$$
\frac{dE}{dt}+D\leq \sum_{k=1}^{10}\,{\mathbf I}_k,
$$
where the energy and the dissipation are given respectively by \eqref{E} and \eqref{D}.

In what follows we will estimate each term ${\mathbf I}_k$ for $k=1,2,\cdots,10$.\\
\noindent{\bf Estimation of ${\mathbf I}_1$}: By Cauchy-Schwarz inequality, and using \eqref{gE}, \eqref{partialggg}, we get
\begin{eqnarray*}
\left|{\mathbf I}_1\right|&\lesssim&\left(1+\|g\|_\infty^4\right)\,\|\partial_y^2\,g\|_2\,\|\partial_y\,g\|_4^2,\\
&\lesssim&\left(1+E^2\right)\,D^{1/2}\, D^{1/4}\,E^{3/4}.
\end{eqnarray*}
This leads to
\begin{equation}
\label{I1}
\left|{\mathbf I}_1\right|\,\lesssim\,\left(E^{3/4}+E^{11/4}\right)\,D^{3/4}.
\end{equation}

\noindent{\bf Estimation of ${\mathbf I}_2$}: Using  \eqref{gE}, we obtain that
\begin{eqnarray}
\nonumber
|{\mathbf I}_2| & =  & |\int_{-1}^1\,\left[(1+g)^5-1\right]\,\left(\partial_y^2\,g\right)^2\,dy|\\
\nonumber
& \lesssim  & \left(\|g\|_\infty+ \|g\|_\infty^5\right)\|\partial_y^2g\|_2^2\\
\label{I2}
&\lesssim  & \left(E^{1/2}+E^{5/2}\right)D.
\end{eqnarray}

\noindent{\bf Estimation of ${\mathbf I}_3$}: We have
$$
|{\mathbf I}_3| \lesssim  \left(1+ \|g\|_\infty^3\right)\|\partial_y^2g\|_2\|\partial_yg\|_6^3.$$
Using \eqref{gE} and \eqref{partialggg} (with $p=6$), we obtain that
\begin{eqnarray}
\nonumber
|{\mathbf I}_3| & \lesssim & \left(1+E^{3/2}\right)D^{1/2}\,E\,D^{1/2}
\\
\label{I3}
& \lesssim & \left(E+E^{5/2}\right)D.
\end{eqnarray}
\noindent{\bf Estimation of ${\mathbf I}_{4}$}:  We have
\begin{eqnarray*}
|{\mathbf I}_{4}|&\lesssim& \left(1+\|g\|_\infty^4\right)\bigg(\int_{-1}^1\,|\partial_y g|(\partial_y^2g)^2dy\bigg) \\&\lesssim&
\left(1+\|g\|_\infty^4\right)\|\partial_y g\|_2\|(\partial_y^2 g)^2\|_2 \\&\lesssim& \left(1+E^2\right) E^{1/2} \|\partial_y^2g\|_4^2.
\end{eqnarray*}
Using \eqref{N} with $p=4,$ we obtain
\begin{equation}
\label{I4}
|{\mathbf I}_{4}|\lesssim \left(E^{1/2}+E^{5/2}\right)\,D.
\end{equation}
\noindent\noindent{\bf Estimation of ${\mathbf I}_5$}: By using an integration by parts, (Note that $(1+g)^5-1=0$ for $y=\pm 1)$ we can write
\begin{eqnarray*}
{\mathbf I}_5 & \thickapprox  &\int_{-1}^1\,\left((1+g)^5-1\right)\left(\partial_y^2g\right)^2\,dy+\int_{-1}^1\,y(1+g)^4\partial_y g\left(\partial_y^2g\right)^2\,dy\\ & \thickapprox  &{\mathbf I}_2+{\mathbf I}_4.
\end{eqnarray*}
Therefore
\begin{equation}
\label{I5}
|{\mathbf I}_5|\lesssim \left(E^{1/2}+E^{5/2}\right)D.
\end{equation}
\noindent{\bf Estimation of ${\mathbf I}_6$}: By using an integration by parts, we can write
\begin{eqnarray*}
{\mathbf I}_6&=&-\frac{1}{10}\int_{-1}^1\,(1-y^2) (1+g)^2\,\partial_y\left[\left(\partial_y\,g\right)^5\right]\,dy\\
&=&-\frac{1}{5}\int_{-1}^1\,y (1+g)^2\,\left(\partial_y\,g\right)^5\,dy\\
&&+\frac{1}{5}\int_{-1}^1\,(1-y^2) (1+g)\,\left(\partial_y\,g\right)^6\,dy.
\end{eqnarray*}
Then, by \eqref{gE} and \eqref{partialggg}
\begin{eqnarray}
\nonumber
|{\mathbf I}_6|& \lesssim  &\left(1+\|g\|_\infty^2\right)\|\partial_y g\|_5^5+\left(1+\|g\|_\infty\right)\|\partial_y g\|_6^6
\\
\label{I6}
& \lesssim  & \left(1+E\right)E^{7/4}D^{3/4}+ \left(1+E^{1/2}\right)E^{2}D\\
\nonumber
&\lesssim&\left(E^{7/4}+E^{11/4}\right)D^{3/4}+\left(E^2+E^{5/2}\right) D.
\end{eqnarray}

\noindent{\bf Estimation of ${\mathbf I}_7$}: { Making use of \eqref{gE} and \eqref{eps}, we can write
\begin{eqnarray}
\nonumber
|{\mathbf I}_7|&\lesssim& \left(1+\|g\|_\infty^3\right)\|(1-y^2)^{1/2}\partial_y^2\,g\|_\infty^2\,\|\partial_y\,g\|_2^2\\
\label{I7}
&\lesssim&\left(1+E^{3/2}\right)\,D\,E=\left(E+E^{5/2}\right)\,D.
\end{eqnarray}
}

\noindent{\bf Estimation of ${\mathbf I}_9$}:  By H\"older's inequality, \eqref{gE}, \eqref{partialggg} and  \eqref{partialgsecond}, we get
\begin{eqnarray}
\nonumber
|{\mathbf I}_9|&\lesssim &(1+\|g\|_\infty^4)\|\partial_y\,g\|_\infty\,\|{\sqrt{1-y^2}}\partial_y^3\,g\|_2\,\|{\sqrt{1-y^2}}\,\partial_y^2\,g\|_2\\
\label{I9}
&\lesssim &(1+E^2)\,D^{1/4}E^{1/4}\,D^{1/2} D^{1/4}E^{1/4}\\
\nonumber
&\lesssim&\left(E^{1/2}+E^{5/2}\right)\,D.
\end{eqnarray}

\noindent{\bf Estimation of ${\mathbf I}_8$}: By integration by parts using the fact that $y=-\partial_y(\frac{1-y^2}{2})$, we have
\begin{equation}
\label{I4789}
{\mathbf I}_4=-\frac{120}{11}\,{\mathbf I}_7-\frac{15}{2}\,{\mathbf I}_8-\frac{60}{7}\,{\mathbf I}_9.
\end{equation}
Hence
\begin{eqnarray}
\nonumber
|{\mathbf I}_8|&\lesssim & |{\mathbf I}_9|+ |{\mathbf I}_7|+|{\mathbf I}_4|\\
\label{I8}&\lesssim & \left(E^{1/2}+E^{5/2}\right)\,D+\left(E+E^{5/2}\right)\,D\\
\nonumber
&\lesssim& \left(E^{1/2}+E+E^{5/2}\right)\,D.
\end{eqnarray}
\noindent{\bf Estimation of ${\mathbf I}_{10}$}: Using integration by parts, we get
\begin{eqnarray*}
{\mathbf I}_{10}&=&-\frac{1}{2}\int_{-1}^1\,(1-y^2)\left[(1+g)^5-1\right]\,\partial_y^4\,g\,\partial_y^2\,g\,dy\\
&\thickapprox &\int_{-1}^1\,y\left[(1+g)^5-1\right]\,\partial_y^3\,g\,\partial_y^2\,g\,dy\\&&+\int_{-1}^1\,(1-y^2)(1+g)^4\,\partial_y^3\,g\,\partial_y^2\,g\partial_y\,g\,dy\\
&&+
\int_{-1}^1\,(1-y^2)\left[(1+g)^5-1\right]\,(\partial_y^3\,g)^2\,dy\\&\thickapprox & {\mathbf I}_5+{\mathbf I}_9 +\int_{-1}^1\,(1-y^2)\left[(1+g)^5-1\right]\,(\partial_y^3\,g)^2\,dy.
\end{eqnarray*}
Hence we obtain thanks to \eqref{gE}, \eqref{I5} and \eqref{I9}
\begin{eqnarray}
\nonumber
|{\mathbf I}_{10}|&\lesssim& |{\mathbf I}_{5}|+|{\mathbf I}_{9}|+\left(\|g\|_\infty+\|g\|_\infty^5\right)D\\
\label{I10}
&\lesssim& |{\mathbf I}_{5}|+|{\mathbf I}_{9}|+\left(E^{1/2}+E^{5/2}\right)D\\
\nonumber
&\lesssim&\left(E^{1/2}+E^{5/2}\right)D.
\end{eqnarray}
Putting all the estimates \eqref{I1}--\eqref{I10} together, we get
$$\frac{dE}{dt}+D\lesssim\left(E^{1/2}+E+E^2+E^{5/2}\right)\,D+
\left(E^{3/4}+E^{7/4}+E^{11/4}\right)\,D^{3/4}.
$$
Using the fact that
\begin{equation}
\label{Ineq}
x^{\alpha_1}+x^{\alpha_2}+\cdots+x^{\alpha_n}\lesssim x^{\alpha_1}+x^{\alpha_n},
\end{equation}
for $\alpha_n>\alpha_{n-1}>\cdots>\alpha_1>0$ and $x\geq 0,$
we conclude that
$$
\frac{dE}{dt}+D\lesssim\left(E^{1/2}+E^{5/2}\right)\,D+
\left(E^{3/4}+E^{11/4}\right)\,D^{3/4}.
$$
By Young's inequality, we have
\begin{eqnarray*}
E^{3/4}\,D^{3/4}=E^{1/2}(E^{1/4}\,D^{3/4})&\leq&\,E^2+E^{1/3}D,\\
E^{11/4}\,D^{3/4}=E^{5/2}(E^{1/4}\,D^{3/4})&\leq&\,E^{10}+E^{1/3}D.
\end{eqnarray*}
Using again \eqref{Ineq} we conclude the proof.
\end{proof}
\section{Proof of the main results}
This section is devoted to the proof of the main results stated in the introduction. We begin by proving Theorem \ref{MainThm}. Using the equivalent expression of the energy $E$ given in \eqref{E}, it suffices to prove the following.
\begin{theo}
\label{MainThmg}
Let $g$ be a global smooth solution of \eqref{tfgg} with initial data $g_0$. There exists $\varepsilon>0$ such that, if
$$
E(0)< \varepsilon,
$$
then there exists $\gamma>0$ such that
$$
E(t)\leq E(0)\,{\rm e}^{-\gamma t},\quad t\geq 0.
$$
\end{theo}
\begin{proof}
The proof follows by using Proposition \ref{EnEs} and Corollary \ref{DecayC} with $\alpha_1=2,\; \alpha_2=10,\; \beta_1={1/3},\; \beta_2={5/2}.$
\end{proof}
\begin{proof}[Proof of Corollary \ref{u0}]
By \eqref{MLV}-\eqref{G}, we have that
$$u(t,Z(t,y))-u^\infty(y)=\frac{1}{2}(1-y^2)g(t,y),\; t\geq 0,\; y\in [-1,1].$$ In particular, for $t=0$ we have
$$u(0,Z(0,y))-u^\infty(y)=\frac{1}{2}(1-y^2)g(0,y)=:\frac{1}{2}(1-y^2)g_0(y),$$ where $g_0$ in the initial data for $g$ and $Z(0,y):=Z_0(y)$ is the initial data for $Z.$ The proof follows by using \eqref{gE}.
\end{proof}

\begin{proof}[{Proof of Corollary \ref{lamb}}]
By \eqref{G}, we have $Z_y(t,y)=\frac{1}{1+g(t,y)}$. Integrating with respect to $y$ and using $Z(t,\pm 1)=\lambda_{\pm}(t)$, we find that
$$
\lambda_+(t)-\lambda_-(t)=\int_{-1}^1\,\frac{dy}{1+g(t,y)}.
$$
It follows from \eqref{Et} and \eqref{gE} that, for $t$ sufficiently large, we have
\begin{eqnarray*}
|\lambda_+(t)-\lambda_-(t)-2|&=&\left|-\int_{-1}^1\,\frac{g(t,y)}{1+g(t,y)}dy\right|\\
&\lesssim&\|g(t)\|_\infty\\
&\lesssim&{\rm e}^{-\frac{\gamma}{2}t},
\end{eqnarray*}
where $\gamma$ is as in \eqref{Et}.
Next, we integrate $(y-\frac{y^3}{3})Z_y$ to get
\begin{eqnarray}
\nonumber
\lambda_+(t)+\lambda_-(t)&=&
\frac{3}{2}\int_{-1}^1\frac{y-\frac{y^3}{3}}{1+g(t,y)}\,dy
+\frac{3}{2}\int_{-1}^1\,(1-y^2)Z(t,y)\,dy\\
\label{aZy}
&=&\frac{3}{2}\int_{-1}^1\frac{y-\frac{y^3}{3}}{1+g(t,y)}\,dy,
\end{eqnarray}
where we have used
$$
\frac{3}{2}\int_{-1}^1\,(1-y^2)Z(t,y)\,dy=\frac{3}{2}\int_{\lambda_-(t)}^{\lambda_+(t)}\,x u(t,x)\,dx=\mu(u(t))=\mu(u_0)=0.
$$
Taking advantage of the fact that $\displaystyle\int_{-1}^1\,\left(y-\frac{y^3}{3}\right)\,dy=0$, we infer
\begin{equation}
\label{sumlam}
\lambda_+(t)+\lambda_-(t)=-\frac{3}{2}\int_{-1}^1\frac{(y-\frac{y^3}{3})g(t,y)}{1+g(t,y)}\,dy.
\end{equation}
We conclude the proof of \eqref{lambda0} by using \eqref{sumlam}, \eqref{Et} and \eqref{gE}.
Combining \eqref{cvlam} and \eqref{lambda0} we obtain \eqref{lambdapm}. This finishes the proof of Corollary \ref{lamb}.
\end{proof}
\begin{proof}[Proof of Corollary \ref{Zt}]
By \eqref{G} and \eqref{defg} we have the equality
$$\partial_y Z(t,y)-\partial_y Z^\infty=\partial_y Z-1=-\frac{g}{1+g}.$$
The proof follows using \eqref{Et} and \eqref{gE}.
\end{proof}

\end{document}